\documentclass[reqno,oneside,11pt]{amsart}

\usepackage{amssymb,mathptmx,eucal,array,enumitem,setspace,color,multirow}
\usepackage{palatino}
\usepackage[T1]{fontenc}
\usepackage{mathtools} 
\usepackage[noadjust]{cite}
\usepackage{geometry}

\usepackage{tikz}
\usetikzlibrary{decorations.pathmorphing,cd} 
\tikzset{>=to}

\usepackage[pdfpagemode=UseNone, pdfstartview={XYZ null null null}]{hyperref}
\hypersetup{
colorlinks=true,
citecolor=red,
linkcolor=blue,
urlcolor=blue
}
\usepackage[capitalise,noabbrev]{cleveref} 

\setlist[itemize]{leftmargin=*}                  
\setlist[enumerate]{leftmargin=*,                
label=\textup{(\roman*)}}    

\DeclareSymbolFont{largesymbols}{OMX}{zplm}{m}{n} 

\let\originalleft\left     
\let\originalright\right
\renewcommand{\left}{\mathopen{}\mathclose\bgroup\originalleft}
\renewcommand{\right}{\aftergroup\egroup\originalright}
\usepackage{marginnote}
\newcommand{\fld}[1]{\mathbb{#1}} 
\newcommand{\ZZ}{\fld{Z}}
\newcommand{\NN}{\fld{N}}

\newcommand{\RR}{\fld{R}}
\newcommand{\CC}{\fld{C}}

\newcommand{\alg}[1]{\mathcal{#1}}  

\DeclareMathOperator{\im}{im}

\DeclareMathOperator{\proj}{proj}

\theoremstyle{plain}
\newtheorem{theorem}{Theorem}[section]
\newtheorem{lemma}[theorem]{Lemma}
\newtheorem{proposition}[theorem]{Proposition}

\newtheorem{definition}[theorem]{Definition}

\newtheorem{remark}[theorem]{Remark}

\newcommand{\bili}[2]{\left\langle #1,#2\right\rangle}
\newcommand{\comm}[2]{\left[ #1,\, #2 \right]} 
\newcommand{\acomm}[2]{\left\lbrace #1,\, #2 \right\rbrace} 
\newcommand{\Lie}[1]{\mathfrak{#1}} 
\newcommand{\osp}{\Lie{osp}}
\newcommand{\normxx}[1]{|x|^{#1}}
\newcommand{\dN}{d}
\newcommand{\Clif}{Cl}
\newcommand{\eps}{\varepsilon} 
\newcommand{\dcover}[1]{\widetilde{#1}}
\newcommand{\dsig}[1]{\dcover{\sigma}_{#1}}

\newcommand{\alphu}{\underline{\alpha}}

\newcommand{\CK}{\mathbf{CK}}
\newcommand{\Poly}{\mathcal{P}}
\newcommand{\Mono}{\mathcal{M}}
\newcommand{\Harmo}{\mathcal{H}}
\DeclareMathOperator{\Pin}{Pin}

\newcommand{\KelvI}{\mathsf{I}_{\kappa}}
\newcommand{\KelvK}{\mathsf{K}_{\kappa}}

\newcommand{\sym}[1]{O_{#1}} 
\newcommand{\Dun}[1]{\mathrm{D}_{#1}}
\newcommand{\DDop}{\underline{\Dun{}}} 
\newcommand{\xun}{\underline{x}} 
\newcommand{\xsq}{|x|^2} 
\newcommand{\Euler}{\mathbb{E}}
\newcommand{\Lapl}{\Delta}
\newcommand{\DLapl}{\Delta_{\kappa}}

\newcommand{\RCA}{\alg{A}_{\kappa}} 

\newcommand{\mm}{\mathfrak{m}}
\newcommand{\zz}{\mathfrak{z}}
\newcommand{\zzpoly}{\mathsf{Z}}
\newcommand{\fatj}{\mathbf{j}}

\newcommand{\base}{\mathcal{B}}
\newcommand{\zzu}{\underline{\zz}}

\newcommand{\xunM}{\xun_{[M]}} 
\newcommand{\DDopM}{\DDop_{[M]}} 
\newcommand{\gammaM}{\gamma_{[M]}}
\newcommand{\EulerM}{\Euler_{[M]}}
\newcommand{\HM}{H_{[M]}}

\newcommand{\zM}[1]{\zz_{#1,[M]}}

\newcommand{\zA}[2]{\zz_{#1,[#2]}}
\newcommand{\zzA}[1]{\zA{#1}{#1}}

\newcommand{\xunA}[1]{\xun_{[#1]}} 
\newcommand{\DDopA}[1]{\DDop_{[#1]}} 

\newcommand{\EulerA}[1]{\Euler_{[#1]}}

\title{Generalised symmetries and  bases for Dunkl monogenics}

\author[H De Bie]{Hendrik De Bie}

\address[Hendrik De Bie]{ Clifford Research Group,
	Department of Electronics and Information Systems,
	Faculty of Engineering and Architecture,
	Ghent University,
	Krijgslaan 281--S8, 9000 Gent, Belgium}
\email{Hendrik.DeBie@UGent.be}

\author[A Langlois-R\'emillard]{Alexis Langlois-R\'emillard}

\address[Alexis Langlois-R\'emillard]{ Department of Applied Mathematics, Computer Science and Statistics, Faculty of Sciences, Ghent University, Krijgslaan 281--S9, 9000 Gent, Belgium.
}
\email{Alexis.LangloisRemillard@UGent.be}

\author[R Oste]{Roy Oste}

\address[Roy Oste]{ Department of Applied Mathematics, Computer Science and Statistics, Faculty of Sciences, Ghent University, Krijgslaan 281--S9, 9000 Gent, Belgium.
}
\email{Roy.Oste@UGent.be}

\author[J Van der Jeugt]{Joris Van der Jeugt}

\address[Joris Van der Jeugt]{ Department of Applied Mathematics, Computer Science and Statistics, Faculty of Sciences, Ghent University, Krijgslaan 281--S9, 9000 Gent, Belgium.}
\email{Joris.VanderJeugt@UGent.be}

\date{June, 27, 2022}


\begin{document}

\begin{abstract}
	We introduce a family of commuting generalised symmetries of the Dunkl--Dirac operator inspired by the Maxwell construction in harmonic analysis. As an application, we use these generalised symmetries to construct bases of the polynomial null-solutions of the Dunkl--Dirac operator. These polynomial spaces form representation spaces of the Dunkl--Dirac symmetry algebra. For the $\ZZ_2^\dN$ case, the results are compared with previous investigations.
\end{abstract}

\maketitle

\onehalfspacing

%
%
\section{Introduction} \label{sec:Intro}

Since their introduction by Dunkl in 1989~\cite{dunkl_differential-difference_1989}, the family of commutative differential-difference operators associated with a reflection group $W$, now known as Dunkl operators, have enjoyed a great deal of interest of mathematical nature and also for applications in physics. 
Due to their properties, it is possible to replace partial derivatives with Dunkl operators in classical differential equations and operators appearing in many physical systems. A great deal of work has been done in the study of the resulting differential operators, most notably on the Dunkl version of the Laplace operator and its harmonic functions. 

This work focuses on the kernel of the Dunkl version of the Dirac operator, which, like its classical analogue, is a square root of the Dunkl Laplacian. 
Polynomials in the kernel of the Dunkl--Dirac operator are called Dunkl monogenics and they form solutions of the Dunkl version of the homogeneous Dirac equation.

The study of the Dunkl--Dirac operator $\DDop$ and its kernel can take many ways. A recent fruitful path to its understanding resides in the consideration of the symmetry algebra linked to the $\osp(1|2)$ realisation generated by $\DDop$ and its dual symbol $\xun$~\cite{orsted_howe_2009}. This symmetry algebra consists of elements supercommuting with $\DDop$ and the Dunkl monogenics  form natural representation spaces. 

The representation theory of the symmetry algebra was studied for few specific reflection groups, namely: $W=\ZZ_2^\dN$~\cite{de_bie_z2n_2016}, where the link was made with the (higher-rank) Bannai--Ito algebra and for any reducible rank 3 reflection group  $W=D_{2m}\times\ZZ_2$~\cite{DBLROVdJ_2020}, where the finite-dimensional representations were constructed. It was also shown that the Dunkl--Dirac symmetry algebra can be considered as a specialisation of an abstract algebra~\cite{de_bie_algebra_2018}.

Symmetries play an important role in our study. We call $S$ a \emph{symmetry} of an operator $A$ if $\comm{S}{A} = SA - AS = 0$; a \emph{generalised symmetry} of  $A$ if $\comm{S}{A} = fA$ for a certain operator $f$, and a \emph{supersymmetry} if $S$ supercommutes with $A$. Finally, $S$ anticommutes with $A$ if $\acomm{S}{A}:= SA+AS = 0$.

The goal of this article is to introduce a class of generalised symmetries of the Dunkl--Dirac operator. Since these generalised symmetries preserve the kernel of $\DDop$, they can be used to construct natural bases for the spaces of monogenic polynomials.

The generalised symmetries are related to the Maxwell representation in harmonic analysis~\cite[p.69]{muller_analysis_1998}, which was translated to Dunkl harmonic analysis by Xu~\cite{xu_harmonic_2000} and to Dunkl--Clifford analysis in~\cite{fei_fueters_2009,yacoub}. Similar operators  were also considered in the study of the conformal symmetries of the super Dirac operator~\cite{coulembier_conformal_2015} and on the radially deformed Dirac operator~\cite{de_bie_new_2017}. The last two were presented via Kelvin inverses; the generalised symmetries defined here are valid also in a more general context of~\cite{de_bie_algebra_2018}, but admit a presentation using a Clifford--Kelvin type transform when specialised to the Dunkl setting.

As an application, we use these generalised symmetries to give a new interpretation of the basis previously obtained by means of a Dunkl version of the Cauchy--Kovalevskaya (CK) extension Theorem in~\cite{de_bie_z2n_2016}.

We now go through the structure of the paper and highlight the main results. In Section~\ref{sec:Dunkl}, we introduce the preliminaries on Dunkl operators and rewrite some results of Xu~\cite{xu_harmonic_2000} on Dunkl harmonics in terms of generalised symmetries of the Dunkl--Laplace operator. Section~\ref{sec:Dunklmono} goes from the Dunkl harmonics to the Dunkl monogenics. We introduce a class of operators and prove their main properties. They are generalised symmetries of the Dunkl--Dirac operator (Proposition~\ref{prop:gensym}), they commute with each other (Proposition~\ref{prop:commZ}), they can be written by means of a Dunkl--Clifford--Kelvin transform (Proposition~\ref{prop:zkelv}) and they are related with a monogenic projection operator (Propositions~\ref{prop:zproj} and~\ref{prop:zandproj}). A basis of the monogenic representation for any reflection group is then constructed in Section~\ref{sec:Monobase} (Theorem~\ref{thm:maxwellbasis}). Finally, we study in Section~\ref{sec:ExAb} the case of the group $W=\ZZ_2^\dN$ and retrieve  a known basis (Proposition~\ref{prop:maxwellandck}).

\section{Dunkl operators} \label{sec:Dunkl}

\subsection{Preliminaries}

Let $W$ be a reflection group acting on $\RR^\dN$ and $\bili{-}{-}$ be the canonical bilinear form of $\RR^\dN$. Let $R \subset \RR^\dN$ denote the root system linked to $W$ and  $R^+$ is a fixed set of positive roots. The reflection $\sigma_{\alpha}: \RR^\dN \to \RR^\dN$ associated with a root $\alpha = (\alpha_1,\dots,\alpha_\dN)\in R^+$ is 
\begin{equation}\label{eq:actionref}
	\sigma_{\alpha}(y) := y - 2\frac{\bili{y}{\alpha}}{\bili{\alpha}{\alpha}}\alpha.
\end{equation}
The group $W$ is generated as a Coxeter group by the reflections $\sigma_{\alpha}$ and its elements act on functions of $x\in \RR^\dN$ by
\begin{equation}
	\omega f(x) = f(\omega^{-1}x), \quad \omega \in W.
\end{equation}

From now on, we will assume that the roots of $R$ are normalised. We consider a $W$-invariant function $\kappa:R\to \CC$.
We will usually assume $\kappa$ to be a positive real function to avoid problems possibly resulting from specific negative values (for example, the definition of the projection operator~\eqref{eq:xu2point5} could potentially have a division by zero for negative $\kappa$). 
Let $\xi_1,\dots, \xi_\dN$ denote the canonical basis of $\RR^\dN$. The Dunkl operator associated with $\xi_j$ is then defined by 
\begin{equation}
	\Dun{j}f(x) = \partial_{x_j}f(x) + \sum_{\alpha\in R^+} \kappa(\alpha) \frac{f(x)-\sigma_{\alpha}f(x)}{\bili{\alpha}{x}} \alpha_j.
\end{equation}
A major, and non-trivial, property of these operators is that they commute~\cite{dunkl_differential-difference_1989}: $\comm{\Dun{j}}{\Dun{k}} = 0$.

With normalised roots, the commutation relations between the Dunkl operators and the variables are given by
\begin{equation}\label{eq:commDx}
	\comm{\Dun{i}}{x_j}= \delta_{ij} + 2\sum_{\alpha\in R^+} \kappa(\alpha) \alpha_i\alpha_j \sigma_{\alpha}, \quad \text{where $\delta_{ij} =1$ if $i=j$, and $0$ otherwise,}
\end{equation}
and one readily sees that $\comm{\Dun{i}}{x_j} = \comm{\Dun{j}}{x_i}$.

The algebra generated by  $x_1,\dots,x_\dN$, $\Dun{1},\dots,\Dun{\dN}$ and the group algebra $\CC W$
is a realisation of the faithful polynomial representation of a rational Cherednik algebra~\cite{rouquier2005representations}. We denote it by $\RCA$, with the index $\kappa$ indicating the Dunkl realisation.

\begin{remark}
	This is an example of the algebra $\alg{A}$ considered in~\cite[Ex.~4.2]{de_bie_algebra_2018}.
\end{remark}

For a multi-index $\beta = (\beta_1,\dotsc,\beta_\dN) \in \NN^\dN$, its 1-norm is $|\beta|_1 := \beta_1 + \dotsb + \beta_\dN$ and we denote the monomial $x^\beta := x_1^{\beta_1} \dotsm x_\dN^{\beta_\dN}$ and also $\Dun{}^\beta := \Dun{1}^{\beta_1} \dotsm \Dun{\dN}^{\beta_\dN}$. 

The Dunkl--Laplace operator $\DLapl$,  the squared norm and the norm are respectively given by 
\begin{align}
	\DLapl &:= \sum_{j=1}^\dN \Dun{j}^2 , & \xsq &:=  \sum_{j=1}^\dN x_j^2, &  |x| &:= \sqrt{\sum_{j=1}^\dN x_j^2},
\end{align}
which are all invariant under the action of $W$. A consequence of this invariance and application of the Dunkl--Leibniz rule for radial functions is the property
\begin{equation}\label{eq:dunklleibniz}
	\comm{\Dun{j}}{\normxx{a}} = a \normxx{a-2} x_j, \quad \text{for } a \in \RR.
\end{equation}

The classical Euler operator $\Euler$, which measures the degree of a homogeneous polynomial, is also $W$-invariant.
A direct computation using~\eqref{eq:commDx} (see for example~\cite{de_bie_algebra_2018}) shows that 
$\DLapl$, $\xsq$ and 
\begin{align}\label{eq:HinDunkl}
	H &:= \frac{1}{2}\sum_{j=1}^\dN \acomm{\Dun{j}}{x_j} = \Euler + \dN/2 + \gamma,&  \text{where}&	&	\Euler&:= \sum_{j=1}^\dN x_j\partial_{x_j}, & \gamma&:= \sum_{\alpha\in R^+} \kappa(\alpha),
\end{align}
form a $\Lie{sl}_2$-triple in the algebra $\RCA$ as the following relations hold:
\begin{align}\label{eq:commsl2}
	\comm{H}{\xsq} &= 2\xsq, & \comm{H}{\DLapl} &= -2\DLapl, & \comm{\DLapl}{\xsq} &=  4H. 
\end{align}
Moreover, we also have the relations
\begin{align}\label{eq:HinDunkl2}
	\comm{H}{x_j} &= x_j, & \comm{H}{\Dun{j}} &= -\Dun{j}, & \comm{\DLapl}{x_j} &= 2 \Dun{j}, & \comm{\xsq}{\Dun{j}} &= -2x_j .
\end{align}

\subsection{Dunkl harmonics}
We will denote by $\Poly = \Poly(\RR^\dN)$ the space of complex-valued polynomials on $\RR^\dN$ and by $\Poly_n = \Poly_n(\RR^\dN)$ the space of homogeneous polynomials of degree $n$. 
The space $\Harmo$ of Dunkl harmonic polynomials consists of all polynomials in the kernel of the Dunkl--Laplace operator $\DLapl$. We further denote $\Harmo_n = \Harmo \cap \Poly_n$.

In a classical construction of harmonic analysis, the Maxwell representation~\cite{muller_analysis_1998} allows one to construct bases of polynomial harmonics by means of the Kelvin transformation.
This was extended by Xu to Dunkl harmonics~\cite{xu_harmonic_2000}. For $\beta = (\beta_1, \dots , \beta_\dN)\in \NN^\dN$, Xu considered the harmonic polynomials $H_{\beta}(x)$ given by (compare with~\cite[Def.~2.2]{xu_harmonic_2000})
\begin{align}\label{eq:Xuharmo}
	H_{\beta}(x) &:= \KelvK\Dun{1}^{\beta_1}\dots \Dun{\dN}^{\beta_\dN}\KelvK(1), 
\end{align}
where a Dunkl version of the Kelvin transform is used
\begin{equation}\label{eq:KelvK}
	\KelvK f (x) :=  \normxx{-(2\gamma+\dN-2)} f\left(\frac{x}{\xsq}\right), \quad \KelvK\KelvK f (x) = f(x).
\end{equation}
It can be verified that they satisfy  $\Lapl_{\kappa} H_{\beta}(x) = 0$.

\subsection{Generalised symmetries}

It is possible to express Xu's construction by means of a generalised symmetry of the Dunkl--Laplace operator. The definition of this operator is inspired by~\cite[Thm~2.3]{xu_harmonic_2000}. 
It is related to the adjoint of a Dunkl operator, see~\cite[Thm~2.1 and Prop.~2.3]{dunkl_differential-difference_1989}.
\begin{definition}
	We define 	$\mm_j\in \RCA$ to be
	\begin{equation}\label{eq:maxwellopH}
		\mm_j = 2x_j (H - 1) - \xsq \Dun{j}.
	\end{equation}
		For a multi-index $\beta=(\beta_1,\dots , \beta_\dN)\in \NN^\dN$, we write $\mm^{\beta} := \mm_1^{\beta_1}\dots \mm_{\dN}^{\beta_\dN}$.
\end{definition}

\begin{proposition}\label{prop:maxgensymH}
	The operator $\mm_j$ is a generalised symmetry of the Dunkl--Laplace operator:
	\begin{equation}\label{eq:maxgensymH}
		\comm{\DLapl}{\mm_j} = 4 x_j \DLapl.
	\end{equation}
\end{proposition}
\begin{proof}
	It follows from the relations~\eqref{eq:commsl2} and~\eqref{eq:HinDunkl2}
	\begin{align*}
		\DLapl \mm_j &= \DLapl (2x_j H - 2 x_j - \xsq\Dun{j})\\
		&= 2 x_j \DLapl H + 4 \Dun{j}H -  2 x_j \DLapl - 4 \Dun{j} - \xsq \Dun{j}\DLapl - 4 H\Dun{j}\\
		&= (2x_jH - 2 x_j - \xsq\Dun{j})\DLapl + 4 x_j \DLapl + 4 \comm{\Dun{j}}{H} - 4 \Dun{j}\\
		&= \mm_j \DLapl + 4 x_j \DLapl.\qedhere
	\end{align*}
\end{proof}
The next result gives the correspondence $\mm^{\beta}(1) = (-1)^n H_{\beta}(x)$ for $\beta\in \NN^d$ with $|\beta|_1=n$.
\begin{proposition}\label{prop:maxDIH}
	For $\beta\in \NN^\dN$ with $|\beta|_1 = n$,
	when acting on $\Poly$,
	\begin{equation}
		\mm_j = - \KelvK \Dun{j}\KelvK, 	\quad \text{and} \quad \mm^\beta = (-1)^n\KelvK \Dun{}^{\beta}\KelvK.
	\end{equation}
\end{proposition}
\begin{proof}
	By linearity, it is sufficient to prove it for a homogeneous polynomial $p\in\Poly_n$. Apply the Dunkl--Leibniz rule~\eqref{eq:dunklleibniz} to get
	\begin{align*}
		\Dun{j}\KelvK p(x) &= \Dun{j} \normxx{-(2\gamma + \dN -2 + 2n)} p(x)\\
		&= \normxx{-(2\gamma + \dN -2 +2n)} \Dun{j}p(x)  - (2\gamma +\dN - 2 + 2n) \normxx{-(2\gamma+\dN + 2n)} x_j p(x).
		\intertext{Both terms have degree of homogeneity $-2\gamma - \dN + 1 - n$; we can apply again the Kelvin transform $\KelvK$ on the two sides to obtain}
		\KelvK\Dun{j}\KelvK p(x) &= \normxx{-(2\gamma + \dN - 2 - 4 \gamma - 2\dN + 2 -2n)}\normxx{-(2\gamma + \dN -2 + 2n)} \Dun{j} p(x)\\
		&\quad- (2\gamma + \dN +2n -2) \normxx{-(2\gamma + \dN -2 - 4\gamma - 2\dN +2 - 2n)}\normxx{-(2\gamma + \dN + 2n)} x_j p(x)\\
		&= \xsq \Dun{j} p(x)  + 2 x_j p(x) - (2\gamma +\dN+2n) x_jp(x),
	\end{align*}
	and this is precisely $-\mm_j p(x) =  - (2x_jH  - 2x_j - \xsq \Dun{j})p(x)$.
\end{proof}

\begin{proposition}\label{prop:commm}
	The generalised symmetries $\mm_j$ commute amongst themselves when acting on $\Poly$
	\begin{equation}
		\comm{\mm_j}{\mm_{k}} = 0.
	\end{equation}
\end{proposition}
\begin{proof}
By Proposition~\ref{prop:maxDIH}, when acting on $\Poly$,
\begin{equation}\label{eq:commm}
	\mm_j\mm_k = \KelvK\Dun{j}\KelvK\KelvK\Dun{k}\KelvK =  \KelvK\Dun{j}\Dun{k}\KelvK = \KelvK\Dun{k}\Dun{j}\KelvK = \KelvK\Dun{k}\KelvK\KelvK\Dun{j}\KelvK = \mm_k\mm_j.\qedhere
\end{equation}
\end{proof}

Let $\proj_{\Harmo}^{\Poly} \colon \Poly \to \Harmo$ denote the projection operator that, when restricted to $\Poly_n$, reduces to $\proj_{\Harmo_n}^{\Poly_n}$ given by \cite[(2.5)]{xu_harmonic_2000}
\begin{equation}\label{eq:xu2point5}
\proj_{\Harmo_n}^{\Poly_n}p(x) = \sum_{j=0}^{\lfloor n/2 \rfloor} \frac{\normxx{2j}\DLapl^{j} p(x)}{2^{2j}j!(-n-d/2-\gamma +2)_j},
\end{equation}
where the notation for the Pochhammer symbol is used, which is defined as $(a)_0=1$, and $(a)_n = a(a+1)\cdots(a+n-1) = \Gamma(a+n)/\Gamma(a)$, with $\Gamma$ the Gamma function.

The Dunkl harmonic $H_\beta (x)=(-1)^n \mm^{\beta}(1) $ is related to the projection~\eqref{eq:xu2point5} as follows.
\begin{theorem}[{\cite[Theorem~2.4]{xu_harmonic_2000}}]\label{thm:xu}
For $\beta\in \NN^\dN$ with $|\beta|_1 = n$, 
\begin{equation}\label{eq:xu}
H_{\beta}(x) = (-1)^n2^n(\gamma -1 + \dN/2)_n\proj_{\Harmo_n}^{\Poly_n}(x^{\beta}).
\end{equation}
\end{theorem}
We conclude this section with a result relating the operator $\mm_j$ with the projection~\eqref{eq:xu2point5}.

\begin{proposition}\label{prop:mproj}
With $H$ given by~\eqref{eq:HinDunkl} and $x_j$ the operator that multiplies a polynomial by $x_j$, when acting on 	$\Harmo$ we have
	\begin{equation}
		\mm_j =  2 (H-2) \circ \proj_{\Harmo}^{\Poly} \circ\, x_j\,.
	\end{equation}
\end{proposition}
\begin{proof}
	Let $h_{n-1}\in\Harmo_{n-1}$, then $\DLapl^{k} x_j(h_{n-1})=0$ for $k\geq 2$, so using~\eqref{eq:xu2point5} we have
	\begin{align*}
2 (H-2)	\proj_{\Harmo}^{\Poly} ( x_j h_{n-1})  
=\ &2 (H-2)\proj_{\Harmo_{n}}^{\Poly_{n}} ( x_j h_{n-1})\\
		=\ &  2 (H-2) x_j h_{n-1} - 2( \Euler + \dN/2 + \gamma-2) \normxx{2}\DLapl (x_j h_{n-1})/(4(\gamma+n-2+d/2))\\
		=\ &  2  x_j(H-1) h_{n-1}  -\normxx{2} \Dun{j} h_{n-1},
	\end{align*}
where we used 	$\comm{\DLapl}{x_j} = 2 \Dun{j}$ and $ \DLapl h_{n-1}=0$.
	The last line is precisely~\eqref{eq:maxwellopH}.
\end{proof}


\section{Dunkl monogenics} \label{sec:Dunklmono}

\subsection{Clifford algebra}

Let $\eps \in \{-1,+1\}$ be a sign and let $\Clif(\dN)$ be the Clifford algebra associated with $\RR^\dN$ and  $\bili{-}{-}$, which is 
generated by $e_1,\dots,e_\dN$, the images of the canonical basis of $\RR^\dN$: $\xi_j\mapsto e_j$,  subject to the following anticommutation relations
\begin{equation}\label{eq:cliffordgen}
	\acomm{e_i}{e_j} = e_ie_j + e_je_i = 2\eps\delta_{ij}. 
\end{equation}
 A double cover $\Pin^{\eps}(\dN)$ of the orthogonal group is realised inside the Clifford algebra by the products of unit vectors. Depending on $\eps$, these groups are in general not isomorphic~\cite{atiyah1964}.

Some specific elements in the tensor product $ \RCA\otimes \Clif(\dN)$ are denoted as follows, with the tensor product omitted,
\begin{align}
	\xun &:= \sum_{j=1}^{\dN} x_je_j,& \DDop &:= \sum_{j=1}^\dN \Dun{j}e_j, & \alphu &:= \sum_{j=1}^\dN\alpha_je_j,
\end{align}
where $\xun$ is the \textit{vector variable} and $\DDop$ the \textit{Dunkl--Dirac operator}. Note that for $\alpha\in\RR^\dN$, $\alphu \in \Clif(\dN)$ which we identify with $ 1\otimes\Clif(\dN) \subset  \RCA\otimes \Clif(\dN)$. 

Up to the sign $\varepsilon$, the square of the Dunkl--Dirac operator is the Dunkl--Laplace operator and its dual operator is the square of the vector variable:
\begin{align}
	\DLapl &:= \sum_{j=1}^\dN \Dun{j}^2 = \varepsilon\DDop^2, & \xsq &:=  \sum_{j=1}^\dN x_j^2 = \varepsilon \xun^2.
\end{align}
Moreover, we have \begin{equation}
	\acomm{\DDop}{\xun} = 2\varepsilon H = 2\varepsilon(\Euler + \dN/2 + \gamma),\label{eq:osprel}
\end{equation}
and $\DDop$ and $\xun$ are the odd generators of a realisation of the Lie superalgebra $\osp(1|2)$ containing the Lie algebra $\Lie{sl}(2)$ as an even subalgebra realised by \eqref{eq:commsl2}~\cite[Theorem~3.4]{de_bie_algebra_2018}:
\begin{equation}
\label{eq:osprel2}
	\comm{\DDop}{\xsq} = 2\xun, \quad \comm{\xun}{\Lapl} = -2\DDop, \quad
\comm{\DDop}{H} = \DDop, \quad \comm{\xun}{H} = -\xun. 
\end{equation}

Unlike $\DLapl$ and $\xsq$, the Dunkl--Dirac operator $\DDop$ and the vector variable $ \xun$ do not commute with the elements of $W$ inside the $\RCA$-part of $\RCA\otimes \Clif(\dN)$, as this copy of $W$ does not interact with $\Clif(\dN)$. However, elements of the form 
\begin{equation}
	\dsig{\alpha} := \alphu\sigma_{\alpha}, \quad\text{for } \alpha\in R,
\end{equation}
anticommute with $\DDop$ and $ \xun$. These elements generate a group
\begin{equation}
	\dcover{W}^{\varepsilon} :=\langle \dsig{\alpha} := \alphu\sigma_{\alpha}\mid \alpha\in R^+\rangle \subset \RCA\otimes\Clif(\dN),
\end{equation}
which is a double covering of the reflection group $W$. Viewing $W$ as a subgroup of the orthogonal group $\mathcal{O}(\dN)$, the double cover $\dcover{W}^{\varepsilon}$ arises through  the pullback of the projection of one of its double coverings $\Pin^{\varepsilon}(\dN)$ onto $\mathcal{O}(\dN)$. An abstract presentation by generators and relations from the ones of $W$ can be also be given, see~\cite{Morris76}.
The dependency of $\dcover{W}^{\varepsilon}$ on $\varepsilon$ is apparent from the inclusion of the Clifford elements in $\alphu$. The two different values of $\varepsilon$ will yield in general the two non-isomorphic double coverings of $W$: $\dcover{W}^+$ and $\dcover{W}^-$. However depending on the group $W$, they may be trivial, as for example for $W=S_3$~\cite[Thm~4.1]{Morris76}.

\begin{lemma}\label{lem:ospgen}
	The Dunkl--Dirac operator $\DDop$ and its dual symbol $\xun$ respect the following relations
	\begin{gather}
		\acomm{\dsig{\alpha}}{\xun} = 0 = \acomm{\dsig{\alpha}}{\DDop};\label{eq:dsiganticomm}\\
		\comm{\Dun{j}}{\xun} = \comm{\DDop}{x_j} = e_j + 2\sum_{\alpha\in R^+} \kappa(\alpha) \alpha_j \dsig{\alpha};\label{eq:DDx}\\
		\comm{\DDop}{\sigma_{\alpha}} = 2\bili{\Dun{}}{\alpha}\dsig{\alpha}, \qquad
		\comm{\xun}{\sigma_{\alpha}} = 2\bili{x}{\alpha}\dsig{\alpha}.
	\end{gather}
\end{lemma}
\begin{proof}
	These relations follow using~\eqref{eq:actionref},~\eqref{eq:commDx} and the Clifford algebra relations. 
\end{proof}

To denote more compactly specific linear combinations of elements in $\dcover{W}^{\varepsilon}$ such as the one appearing in the right-hand side of~\eqref{eq:DDx}, we write (see also~\cite[Ex. 4.2]{de_bie_algebra_2018}) 
\begin{equation}\label{eq:symi2}
	\sym{j}:= \frac{\varepsilon}{2}( \comm{\DDop}{x_j} - e_j) = \varepsilon\sum_{\alpha \in R^+} \kappa(\alpha) \alpha_j \dsig{\alpha}.
\end{equation}

\begin{lemma}\label{lem:ojejisomega}
	The following holds
	\begin{equation}
		\sum_{j=1}^\dN \sym{j} e_j =\sum_{\alpha\in R^+} \kappa(\alpha)\sigma_{\alpha}.
	\end{equation}
\end{lemma}
\begin{proof}
	Replacing $\sym{j}$ by its expression~\eqref{eq:symi2}, and using the anticommutation of Clifford elements and the fact that the roots are normalised, in particular that $\alphu^2 = \eps$, yields
	\begin{align}
		\sum_{j=1}^\dN \sym{j}e_j &= \eps\sum_{j=1}^\dN \sum_{\alpha\in R^+} \kappa(\alpha) \alpha_j \alphu\sigma_{\alpha} e_j = \eps \sum_{\alpha\in R^+} \kappa(\alpha) \alphu^2\sigma_{\alpha} = \sum_{\alpha\in R^+} \kappa(\alpha) \sigma_{\alpha}.
	\end{align}
\end{proof}

\subsection{Dunkl monogenics} 

Let $V$ be an irreducible representation of $\Clif(\dN)$, also called a spinor representation. There is a natural action of $\RCA\otimes \Clif(\dN)$ on the space $\Poly \otimes V$. The space of Dunkl monogenic polynomials consists of the elements of $\Poly \otimes V$ that are in the kernel of the Dunkl--Dirac operator, and will be denoted by $\Mono := \Mono(\RR^\dN;V) $. We denote $\Mono_n :=\Mono_n(\RR^\dN;V)= \Mono \cap (\Poly_n \otimes V)$ for the $\Mono$-subspace of (spinor valued) homogeneous polynomials of degree $n$, and we have $\Mono(\RR^\dN;V) = \bigoplus_{n\geq 0} \Mono_n(\RR^\dN;V)$. 

There is a projection $\proj^{\Poly\otimes V}_{\Mono}: \Poly(\RR^\dN)\otimes V \longrightarrow  \Mono(\RR^\dN;V)$ that, when restricted to $\Poly_n\otimes V$, is given by~\cite[Lem.~4.6]{orsted_howe_2009}
\begin{equation}\label{eq:projpi}
	\begin{gathered}
		\proj^{\Poly_n\otimes V}_{\Mono_n}: \Poly_n(\RR^\dN)\otimes V \longrightarrow  \Mono_{n}(\RR^\dN;V)\\
		p\longmapsto  p - \eps\sum_{j=0}^{\lfloor (n+1)/2\rfloor} \frac{(-1)^{j}\xun^{2j+1}\,\DDop^{2j+1}p}{2^{2j+1}j!(n-j-1+\dN/2+\gamma)_{j+1}} + \sum_{j=1}^{\lfloor n/2\rfloor}\frac{(-1)^j\normxx{2j}\DLapl^j p}{2^{2j}j!(n-j+\dN/2+\gamma)_j}.
	\end{gathered}
\end{equation}
\begin{remark}
	Each Dunkl monogenic polynomial is a highest weight vector for the $\Lie{osp}(1|2)$ realisation containing the Dunkl--Dirac operator as positive root vector. For the $\Lie{osp}(1|2)$ extremal projector $\proj^{\Poly\otimes V}_{\Mono}$, we have~\cite[(3.8a)]{berezin_group_1981} 
	\begin{equation}\label{eq:projproj}
		\proj_{\Mono}^{\Poly\otimes V} =\proj_{\Mono}^{\Harmo\otimes V} \proj_{\Harmo\otimes V}^{\Poly\otimes V},
	\end{equation}
	where $\proj_{\Harmo\otimes V}^{\Poly\otimes V} = \proj_{\Harmo}^{\Poly}$ as considered above~\eqref{eq:xu2point5} for $\Lie{sl}(2)$, and $\proj_{\Mono}^{\Harmo\otimes V}$, when restricted to $\Poly_n\otimes V$, is given by
	\begin{equation}\label{eq:projHM}
	\proj_{\Mono_n}^{\Harmo_n\otimes V} = \left(1- \frac{ \eps\xun\,\DDop}{2(n-1 + \dN/2 + \gamma)}\right).
\end{equation}
\end{remark}

\subsection{Generalised symmetries}

Recall that $\varepsilon\in \{-1,+1\}$ and that the Clifford generators $e_j$ satisfy relations~\eqref{eq:cliffordgen}.

\begin{definition}
	We define $\zz_j\in \alg{A}\otimes\Clif(\dN)$ for $1\leq j\leq \dN$ by 
	\begin{equation}\label{eq:Z}
		\zz_j:= 2\varepsilon x_jH - \xun\Dun{j}\xun.
	\end{equation}
	For a multi-index $\beta=(\beta_1,\dots , \beta_\dN)\in \NN^\dN$, we write $\zz^{\beta} := \zz_1^{\beta_1}\dots \zz_{\dN}^{\beta_\dN}$.
\end{definition}

We begin by giving alternative formulations of $\zz_j$ in terms of elements of $\RCA\otimes \Clif(\dN)$ that follow from Lemma~\ref{lem:ospgen} and the expressions~\eqref{eq:symi2} of $\sym{j}$.
\begin{lemma}\label{lem:zotherversions}
	The operator $\zz_j:= 2\varepsilon x_j H - \xun \Dun{j}\xun$ has the following expressions 
	\begin{align}\label{eq:zzj}
		\zz_j &= x_j\acomm{\DDop}{\xun}  - \xun  \comm{\DDop}{x_j}  - \varepsilon\xsq \Dun{j};\\
		\zz_j &= 2\varepsilon x_j(\Euler + \dN/2 + \gamma)  - \xun (e_j + 2\varepsilon \sym{j}) - \varepsilon\xsq \Dun{j}.\label{eq:lem:zotherversions3}
	\end{align} 
\end{lemma}

We now consider some of the main properties of $\zz_j$ that are useful for our purposes.

\begin{proposition}\label{prop:gensym}
	The operator $\zz_j$ is a generalised symmetry of the Dunkl--Dirac operator
	\begin{equation}
		\comm{\DDop}{\zz_j} = 2\varepsilon x_j\DDop.	
	\end{equation}
\end{proposition}
\begin{proof}
	First we anticommute $\DDop$ and $\xun$ and commute $\DDop$ and $x_j$ by~\eqref{eq:osprel2}
	\begin{align*}
		\DDop \,\zz_j &= \DDop(x_j\acomm{\DDop}{\xun} - \xun\Dun{j}\xun) = (x_j\DDop + \comm{\DDop}{x_j})\acomm{\DDop}{\xun} - (-\xun\,\DDop + \acomm{\xun}{\DDop})\Dun{j}\xun,
		\intertext{now we employ~\eqref{eq:DDx} and the commutation of $\Dun{j}$}
		&= x_j\DDop\acomm{\DDop}{\xun} + \comm{\Dun{j}}{\xun}\acomm{\DDop}{\xun}  +\xun\Dun{j}\DDop\,\xun - \acomm{\DDop}{\xun}\Dun{j}\xun,
		\intertext{then we apply~\eqref{eq:HinDunkl2}, and anticommute a second time $\DDop$ and $\xun$} 
		&= x_j\acomm{\DDop}{\xun}\DDop + 2\varepsilon x_j\DDop + \comm{\Dun{j}}{\xun}\acomm{\DDop}{\xun} - \xun\Dun{j}\xun\,\DDop + \xun\,\Dun{j}\acomm{\DDop}{\xun} - \acomm{\DDop}{\xun}\Dun{j}\xun,
		\intertext{finally, $\acomm{\DDop}{\xun}$ commutes with $\Dun{j}\xun$ because of $\comm{H}{\Dun{j}} = -\Dun{j}$, $\comm{H}{\xun} = \xun$ and $\varepsilon^2=1$ so}
		&= (x_j\acomm{\DDop}{\xun} - \xun\Dun{j}\xun + 2\varepsilon x_j)\DDop  + \comm{\Dun{j}}{\xun}\acomm{\DDop}{\xun} - \Dun{j}\xun \acomm{\DDop}{\xun} + \xun\Dun{j}\acomm{\DDop}{\xun}\\
		&= \zz_j\DDop + 2\varepsilon x_j\DDop.\qedhere
	\end{align*}
\end{proof}

\begin{proposition}\label{prop:relofzandother}
	The operator $\zz_k$ respects the following commutation relations 
	\begin{align}
		\comm{\xun}{\zz_k} &= -2\varepsilon x_k\xun + \xun(e_k + 2\varepsilon\sym{k})\xun;\label{eq:zelxun}\\
		\comm{x_j}{\zz_k} &= -2\varepsilon x_jx_k - \xun\comm{x_j}{\Dun{k}}\xun;\label{eq:zrelx}\\
		\comm{e_j}{\zz_k} &= 2\varepsilon(\xun \Dun{k}x_j - x_j\Dun{k}\xun); \label{eq:zrele}\\
		\comm{\Dun{j}}{\zz_k} &= 2\varepsilon (x_k\Dun{j} - x_j\Dun{k})
		+ 2\varepsilon\comm{\Dun{j}}{x_k} H + e_j\comm{\xun}{\Dun{k}} - 2\varepsilon(\sym{j}\Dun{k}\xun + \xun \Dun{k} \sym{j});\label{eq:zrelD}\\
		\dsig{\alpha} \zz_k &=  \zz_{\sigma_{\alpha}(\xi_k)} \dsig{\alpha}, \quad \text{with} \quad \zz_{\sigma_{\alpha}(\xi_k)}:= \sum_{j=1}^d \bili{\sigma_{\alpha}(\xi_k)}{\xi_j}\zz_j.\label{eq:prop:relofzandothersig}
	\end{align}
\end{proposition}
\begin{proof}
	Equation~\eqref{eq:zelxun} follows from a small calculation using equations~\eqref{eq:HinDunkl2} and~\eqref{eq:osprel2}
	\begin{align*}
		\comm{\xun}{\zz_k} &= 2\varepsilon \xun x_k H - \xun\,\xun\Dun{k}\xun = 2\varepsilon x_k\xun H - \xun\Dun{k}\xun\,\xun - \xun\comm{\xun}{\Dun{k}}\xun - \zz_k\xun = -2\varepsilon x_k \xun + \xun(e_k+2\varepsilon\sym{k})\xun.
	\end{align*}
	
	Equation~\eqref{eq:zrelx} follows from the commutation relation~\eqref{eq:HinDunkl2} between $x_j$ and $H$:
	\begin{align*}
		\comm{x_j}{\zz_k} &= 2\varepsilon x_jx_k H - x_j\xun\Dun{k}\xun - 2\varepsilon \zz_kx_j= 2\eps x_kH x_j -2\varepsilon x_kx_j - (\xun\Dun{k}x_j\xun + \xun \comm{x_j}{\Dun{k}}\xun) - \zz_kx_j\\ &= -2\varepsilon x_jx_k - \xun\comm{x_j}{\Dun{k}}\xun.
	\end{align*}
	
	Equation~\eqref{eq:zrele} comes from $\acomm{e_j}{\xun} = 2\varepsilon x_j$:
	\begin{align*}
		e_j\zz_k &= 2\varepsilon e_j x_k H - e_j\xun\Dun{k}\xun = 2\varepsilon x_k H e_j + \xun e_j\Dun{k} \xun - 2\varepsilon x_j\Dun{k}\xun = \zz_ke_j + 2\varepsilon \xun\Dun{k} x_j - 2\varepsilon x_j\Dun{k}\xun.
	\end{align*}
	
	Slightly more tedious computations yield equation~\eqref{eq:zrelD}. First develop
	\begin{align*}
		\comm{\Dun{j}}{\zz_k} &=  2\varepsilon\Dun{j}x_k H - \Dun{j}\xun\Dun{k}\xun - \zz_k\Dun{j}\\
		&= 2\varepsilon x_k\Dun{j}H + 2\varepsilon \comm{\Dun{j}}{x_k}H - (\xun\Dun{j}\Dun{k}\xun + \comm{\Dun{j}}{\xun}\Dun{k}\xun) \\
		&\quad - (2\varepsilon x_kH\Dun{j} - \xun\Dun{k}\Dun{j}\xun - \xun\Dun{k}\comm{\xun}{\Dun{j}}),
		\intertext{then employ $\comm{H}{\Dun{j}} = -\Dun{j}$ and cancel some terms}
		&=2\varepsilon x_k\Dun{j} + 2\varepsilon\comm{\Dun{j}}{x_k}H - \comm{\Dun{j}}{\xun}\Dun{k}\xun + \xun\Dun{k}\comm{\xun}{\Dun{j}}\\
		&= 2\varepsilon x_k\Dun{j} +2\varepsilon \comm{\Dun{j}}{x_k}H  -(e_j + 2\varepsilon\sym{j})\Dun{k}\xun - \xun \Dun{k}(e_j+2\varepsilon\sym{j}),
		\intertext{now use $\acomm{\xun}{e_j}= 2\varepsilon x_j$ to get}
		&= 2\varepsilon x_k\Dun{j} + 2\varepsilon \comm{\Dun{j}}{x_k}H - e_j \Dun{k}\xun + e_j \xun\Dun{k} - 2\varepsilon x_j \Dun{k} - 2\varepsilon(\sym{j}\Dun{k}\xun + \xun \Dun{k}\sym{j})\\
		&= 2\varepsilon (x_k\Dun{j} - x_j\Dun{k}) + 2\varepsilon\comm{\Dun{j}}{x_k}H + e_j \comm{\xun}{\Dun{k}} - 2\varepsilon(\sym{j}\Dun{k}\xun + \xun \Dun{k}\sym{j}).
	\end{align*}
	Relation~\eqref{eq:prop:relofzandothersig} follows from~\eqref{eq:dsiganticomm}, and from the action of $W$ on $x_j$ and $\Dun{j}$.
\end{proof}

\begin{lemma}\label{lem:vectorzz}
The operator $\zzu:= \sum_{j=1}^\dN \zz_j e_j$ can be written as
\begin{equation}
\zzu = 2\eps \xun(\Euler +\gamma - \sum_{\alpha\in R^+} \kappa(\alpha)\sigma_{\alpha}) - \eps \xsq \DDop.
\end{equation}
\end{lemma}
\begin{proof}
We express $\zz_j$ by~\eqref{eq:lem:zotherversions3} and use Lemma~\ref{lem:ojejisomega}:
\begin{align*}
\zzu &= \sum_{j=1}^\dN (2\eps x_j(\Euler + \dN/2 + \gamma) - \xun(e_j + 2\eps \sym j ) - \eps \xsq \Dun{j})e_j\\
&= 2\eps \xun(\Euler + \dN/2 + \gamma) - \eps \xun \dN - 2\eps\xun \sum_{j=1}^{\dN} \sym j e_j - \eps \xsq \DDop = 2\eps \xun(\Euler +\gamma - \sum_{\alpha\in R^+} \kappa(\alpha)\sigma_{\alpha}) - \eps \xsq \DDop. \qedhere
\end{align*}
\end{proof}

\subsection{Kelvin transformation}

Define the Dunkl--Clifford--Kelvin transform $\KelvI$ as
\begin{equation}\label{eq:CDK}
	\KelvI f(x) = \xun \,\normxx{-(2\gamma + \dN)} f\left(\frac{x}{\xsq}\right).
\end{equation}
Since $\kappa \geq 0$, the sum over positive roots $\gamma$ is non-negative and $\normxx{-(2\gamma + \dN)}$ is thus well-defined. The operator $\KelvI$ is $\varepsilon$-idempotent, that is $\KelvI^2=\varepsilon$. Indeed, using $\xun\, \xun = \eps\normxx{2}$, 
\begin{equation}
	\KelvI\KelvI f(x) = \KelvI\left( \xun \normxx{-(2\gamma + \dN)} f\left(\frac{x}{\xsq}\right)\right)= \xun \normxx{-(2\gamma+\dN)} \left(\frac{\xun}{\xsq} \frac{\normxx{2(2\gamma +\dN)}}{\normxx{(2\gamma +\dN)}} f\left(\frac{x}{\xsq}\frac{\normxx{4}}{\xsq} \right)\right) 
	= \varepsilon f(x).
\end{equation}

The relation between the two Kelvin-type transforms~\eqref{eq:KelvK} and \eqref{eq:CDK} is 
\begin{equation}\label{eq:KelvKandKelvI}
	\KelvI f = \eps \xun \normxx{-2} \KelvK f.
\end{equation}
Remark that for $p(x)\in \Poly_n(\RR^\dN)$ we have $p(x/\xsq) = \normxx{-2n}p(x)$, and thus the action of the Dunkl--Clifford--Kelvin transform becomes
\begin{align}
	\KelvI p(x) &= \normxx{-(2\gamma +\dN + 2n)}\xun p(x).
\end{align}

The transform~\eqref{eq:CDK} was considered before, for example see~\cite{yacoub} and~\cite{fei_fueters_2009}. One of the main results of those two papers is to prove that, for any polynomial monogenic $f$, also $\KelvI\Dun{j}\KelvI(f)$ is a polynomial monogenic. We give an interpretation in terms of generalised symmetries of the Dunkl--Dirac operator.

\begin{proposition}\label{prop:zkelv}
For $\beta\in \NN^\dN$ with $|\beta|_1 = n$, when acting on $\Poly\otimes V$,
	\begin{equation}\label{eq:coro:polyKI}
		\zz_j = -\KelvI\Dun{j}\KelvI,\quad\text{and} \quad \zz^{\beta}= (-1)^n\eps^{n-1} \KelvI\Dun{}^{\beta}\KelvI.
	\end{equation}
\end{proposition}
\begin{proof}
	Let $p\in \Poly_n(\RR^\dN)$ be a homogeneous polynomial of degree $n$. Apply equation~\eqref{eq:dunklleibniz} to get
	\begin{align*}
		\Dun{j}\KelvI p(x) & = \Dun{j}\normxx{-(2\gamma +\dN + 2n)}\xun p(x)\\
		&=  -(2\gamma + \dN + 2n)\normxx{-(2\gamma +\dN + 2n +2)}x_j\xun p(x) + \normxx{-(2\gamma +\dN + 2n)}\Dun{j}\xun p(x).
	\end{align*}
	Remark now that the first and second terms have degree of homogeneity $-2\gamma - \dN - n$. Thus applying another time the Dunkl--Clifford--Kelvin transform yields
	\begin{align*}
		\KelvI\Dun{j}\KelvI p(x) & = -(2\gamma + \dN + 2n)\xun \normxx{-(2\gamma +\dN - 4\gamma -2\dN - 2n)}\normxx{-(2\gamma +\dN + 2n +2)}x_j\xun p(x)\\
		&\qquad+ \xun \normxx{-(2\gamma +\dN - 4\gamma - 2\dN - 2n)}\normxx{-(2\gamma +\dN + 2n)}\Dun{j}\xun p(x)\\
		&= -(2\gamma + \dN + 2n)\xun^2 \normxx{-2} x_j p(x) + \xun\Dun{j}\xun p(x)\\
		&= -2\varepsilon (n + \dN/2 + \gamma) x_j p(x) + \xun \Dun{j} \xun p(x),
	\end{align*}
	which equals $-\zz_j p(x) = -(2\varepsilon x_j(\Euler + \dN/2 + \gamma)  - \xun\Dun{j}\xun)p(x)$.
\end{proof}

The commutativity of the Dunkl operators then implies directly that of the $\zz_j$.
\begin{proposition}\label{prop:commZ}
	The operators $\zz_j$ commute amongst themselves, when acting on $\Poly\otimes V$, namely
	\begin{equation}\label{eq:commz}
		\comm{\zz_j}{\zz_{\ell}} = 0.
	\end{equation}
\end{proposition}
\begin{proof}
Apply Proposition~\ref{prop:zkelv} and the $\eps$-idempotence of $\KelvI$:
\begin{equation}
 	\zz_j\zz_{\ell} = \KelvI\Dun{j}\KelvI\KelvI\Dun{\ell}\KelvI = \eps \KelvI\Dun{j}\Dun{\ell}\KelvI = \eps\KelvI\Dun{\ell}\Dun{j}\KelvI = \eps^2\KelvI\Dun{\ell}\KelvI\KelvI\Dun{j}\KelvI = \zz_{\ell}\zz_j.\qedhere
 \end{equation}
\end{proof}

\begin{remark}
Proposition~\ref{prop:commZ} also holds in $\RCA\otimes \Clif(\dN)$. Furthermore, 
it even holds in the general abstract setting considered in~\cite{de_bie_algebra_2018}. It follows by direct computations using~\eqref{eq:HinDunkl2} and the $\osp(1|2)$ relations~(\ref{eq:osprel}--\ref{eq:osprel2}).  
\end{remark}

\subsection{Projection operator relation}

\begin{proposition}\label{prop:zproj}
With $H$ given by~\eqref{eq:HinDunkl} and $x_j$ the operator that multiplies a polynomial by $x_j$, when acting on 	$\Mono$, we have
	\begin{equation}
		\zz_j = 2\eps  (H-1) \circ \proj_{\Mono}^{\Poly\otimes V} \circ\, x_j\,.
	\end{equation}
\end{proposition}
\begin{proof}
	Let $M_{n}\in\Mono_{n}$, then $\DDop^{k} x_jM_{n}=0$ for $k\geq 3$, since $\DDop^3 x_j M_{n} = \eps\DDop\Lapl_{\kappa} x_jM_{n} = \DDop x_j \Lapl_{\kappa} M_{n} - \DDop\Dun{j}M_{n} =0$, so using~\eqref{eq:projpi} we have
	\begin{align*}
2\eps  (H-1)	\proj_{\Mono}^{\Poly\otimes V} ( x_j M_{n})
&=  2\eps  (\Euler + \dN/2 + \gamma-1) \proj_{\Mono_{n+1}}^{\Poly_{n+1}\otimes V} ( x_j M_{n})\\
		&=  2\eps  (n + \dN/2 + \gamma) x_j M_{n} -  \xun\,\DDop (x_j M_{n}) - \eps\normxx{2}\DLapl (x_j M_{n})/2\\
		&=  2\eps  (n + \dN/2 + \gamma) x_jM_{n}  -  \xun\comm{\DDop}{x_j}M_{n} -\eps\normxx{2}\Dun{j} M_{n},
	\end{align*}
where we used 	$\comm{\DLapl}{x_j} = 2 \Dun{j}$ and $\DDop M_{n} = 0 = \DLapl M_{n-1}$.
	The last line is precisely~\eqref{eq:zzj} acting on $M_n$.
\end{proof}

The next result is related to \cite[Prop~4.2]{yacoub} where the Dunkl--Clifford--Kelvin transform is used.
	\begin{proposition}\label{prop:zandproj}
		Let $\beta \in \NN^\dN$ with $|\beta|_1 =n$ and $x^{\beta} \in \Poly_n$. Then, acting on $ V$
		\begin{equation}
		\zz^{\beta}=  \eps^n2^n(\gamma + \dN/2)_n		\proj_{\Mono_n}^{\Poly_n\otimes V}\circ\,x^{\beta}.
		\end{equation}
	\end{proposition}
	\begin{proof}
	By \eqref{eq:projproj}, we write 
	\begin{equation}
	\proj_{\Mono_n}^{\Poly_n\otimes V} =\proj_{\Mono_n}^{\Harmo_n\otimes V} \proj_{\Harmo_n\otimes V}^{\Poly_n\otimes V} .
	\end{equation}
	Let $s\in V$, then by \eqref{eq:xu} (\cite[Theorem~2.4]{xu_harmonic_2000})
	\[
	\proj_{\Harmo_n\otimes V}^{\Poly_n\otimes V} (x^{\beta}s )=(-1)^n H_\beta \, s /( 2^n(\gamma -1 + \dN/2)_n).
	\]
	Next, we apply \eqref{eq:projHM} and make use of~\eqref{eq:dunklleibniz} 
		\begin{align*}
				\proj_{\Mono_n}^{\Harmo_n\otimes V}  H_{\beta} s &= H_{\beta}s -  \frac{ \eps}{2n +\dN + 2\gamma -2}\xun \,\DDop\KelvK \Dun{}^{\beta} \KelvK(s) \\
			&=  H_{\beta}s - \frac{ \eps}{2n +\dN + 2\gamma -2} \xun \, \DDop  \normxx{2\gamma + \dN - 2 + 2n} \Dun{}^{\beta} \normxx{-2\gamma - \dN + 2}s\\
			&=  H_{\beta}s- \frac{ \eps}{2n +\dN + 2\gamma -2}(2\gamma + \dN - 2 + 2n) \xun\, \xun \normxx{-2}  \normxx{2\gamma + \dN - 2 + 2n} \Dun{}^{\beta} \normxx{-2\gamma - \dN + 2}s\\
			&\quad + \frac{ \eps}{2n +\dN + 2\gamma -2} \xun   \normxx{2\gamma + \dN - 2 + 2n} \Dun{}^{\beta}\DDop \normxx{-2\gamma - \dN + 2} s\\
			&=  H_{\beta}s-  \normxx{2\gamma + \dN - 2 + 2n} \Dun{}^{\beta} \normxx{-2\gamma - \dN + 2}s\\
			&\quad + \frac{ 2\eps(\gamma + \dN/2 -1)}{2n +\dN + 2\gamma -2} \xun   \normxx{2\gamma + \dN - 2 + 2n} \Dun{}^{\beta} \xun \normxx{-2}\normxx{-2\gamma - \dN+2}s\\
			&=  \frac{ \eps(\gamma + \dN/2 -1)}{n +\dN/2 + \gamma -1}\KelvI \Dun{}^{\beta} \KelvI(s),
		\end{align*}
		where we used~\eqref{eq:KelvKandKelvI}. 
		The result now follows by Proposition~\ref{prop:zkelv}. 
\end{proof}


\section{Monogenic bases} \label{sec:Monobase}
The goal of this section is to give a basis for the polynomial monogenics using the generalised symmetries of the previous section. Note that we will here also assume $\kappa$ to be a positive real function.

The idea is that acting with a generalised symmetry on a monogenic polynomial gives a monogenic polynomial of a higher degree. We can thus construct a generating set for $\Mono_n$ starting from the space of monogenics of degree 0. 
The strategy we employ to reduce the generated set to a basis is inspired by the one applied by Xu in the Dunkl harmonic case~\cite{xu_harmonic_2000}. 
Note that for the case of harmonic polynomials, the degree 0 harmonics $\Harmo_0$ are the constant polynomials, while the space of monogenics of degree 0 is given by the spinor space: $\Mono_0 = V$.

For each multi-index $ \beta = (\beta_1,\dots,\beta_\dN)\in \NN^\dN$ with $| \beta|_1=n$ and each spinor $s \in V = \Mono_0$,  we define a polynomial monogenic of degree $n$ by
\begin{equation}
	\zzpoly^{ \beta}_s := \zz^{ \beta} s = \zz_1^{\beta_1} \dots \zz_{\dN}^{\beta_{\dN}} s.
\end{equation}
It is direct to see from~\eqref{eq:coro:polyKI} that
	\begin{equation}
		\zzpoly^{ \beta}_s = (-1)^n\eps^{n-1} \KelvI\Dun{}^{ \beta}\KelvI (s) = (-1)^n\eps^{n-1}\KelvI\Dun{1}^{\beta_1}\dots \Dun{\dN}^{\beta_\dN}\KelvI( s).
\end{equation}

Recall that $\xi_j = (0,\dots, 1,0,\dots,0)\in \NN^\dN \subset \RR^\dN$ and so $\beta + \xi_j = (\beta_1,\dots, \beta_j+1,\dots,\beta_d)$. 
\begin{lemma}\label{lem:recpolyz}
	Let $ \beta\in \NN^\dN$ with $| \beta|_1 = n$ and $s\in V$. The spinor-valued polynomial $\zzpoly^{ \beta}_s$ respects
	\begin{equation}
		\zzpoly^{ \beta + \xi_j}_s = (2\varepsilon( n+\dN/2 +\gamma) x_j  - 2\varepsilon \xun \sym{j} - \xun e_j - \varepsilon \xsq\Dun{j}) \zzpoly^{ \beta}_s.
	\end{equation}
\end{lemma}
\begin{proof}
It follows from the commutativity of the $\zz_j$ along with  their expression~\eqref{eq:lem:zotherversions3}.
\end{proof}

We now turn our attention to the construction of  bases for the monogenics. 

\begin{proposition}
		Let $\nu$ be a basis of $V$, the spinor representation of $\Clif(\dN)$. The set 
	\begin{equation}
		\mathcal{C}_n = \{\zzpoly^{\beta}_s\mid \beta\in\NN^{\dN}, |\beta|_1 = n,\ s\in \nu\ \}
	\end{equation}
	is a generating set for $\Mono_n(\RR^\dN;V)$. Moreover, the following relations hold
	\begin{equation}\label{eq:rellindep}
		\sum_{j=1}^\dN \zzpoly^{\eta + \xi_j}_{e_j\cdot s} = 0, \quad \text{for every}\quad \eta\in \NN^{\dN}, \text{ with } |\eta|_1 = n-1, \ \text{and } s\in V.
	\end{equation}
\end{proposition}
\begin{proof}
	Proposition~\ref{prop:gensym} states that every $\zz_j$ is a generalised symmetry of $\DDop$ and thus $\DDop\zzpoly^{ \beta}_s =0$, so $\operatorname{span}_{\CC}(\mathcal{C}_n) \subset \Mono_n(\RR^d;V)$. The monomials $x^{\beta}$ for $\beta\in \NN^d$ with $|\beta|_1=n$, together with the spinors $s\in \nu$, form a basis of $\Poly_n\otimes V$. The projection $\proj_{\Mono_n}^{\Poly_n\otimes V}$ sends the element $x^{\beta}\otimes s$ to a multiple of $\zzpoly_s^{\beta} = \zz^{\beta}s\in \mathcal{C}_n$  by Proposition~\ref{prop:zandproj}, so $\operatorname{span}(\mathcal{C}_n)$ contains $\im( \proj_{\Mono_n}^{\Poly_n\otimes V})$. The projection is surjective, therefore $\operatorname{span}_{\CC}(\mathcal{C}_n) \supset \Mono_n(\RR^d;V)$. We have thus shown that the set $\mathcal{C}_n$ generates $\Mono_n(\RR^d;V)$.
	
	We now show that the 
	relations of the form~\eqref{eq:rellindep} hold. Let $ \eta\in \NN^{\dN}$ with $|\eta|_1 =n-1$ and $s\in V$. We use now Lemma~\ref{lem:vectorzz} and the relations $\Euler s = 0$,  $\sigma_{\alpha} s = s$ and $\DDop s = 0$ to find
	\begin{align*}
		\sum_{j=1}^\dN \zzpoly^{ \eta + \xi_j}_{e_j\cdot s} &= \sum_{j=1}^\dN \zz^{ \eta}\zz_je_j s = \zz^{ \eta} \zzu s = \zz^{\eta}(2\eps \xun(\Euler +\gamma - \sum_{\alpha\in R^+} \kappa(\alpha)\sigma_{\alpha}) - \eps \xsq \DDop)s = 0.\qedhere
	\end{align*}
\end{proof}

\begin{remark}
	Let $\eta\in \NN^{\dN}$ with $|\eta|_1 = n-2$. Xu  exhibited the relations satisfied by the harmonics of equation~\eqref{eq:Xuharmo}~\cite[p.~500]{xu_harmonic_2000} (see also~\cite[pp. 212-213]{dunkl2014orthogonal}):
	\begin{align}\label{eq:relxu}
		\sum_{j=1}^\dN H_{\eta + 2\xi_j} &= \KelvK \Dun{}^{\eta}\DLapl \KelvK(1) = 0.
	\end{align}
	The monogenics satisfy the same relation, as can be seen by applying twice relation~\eqref{eq:rellindep}, or by viewing it in the Dunkl--Clifford--Kelvin transform
	\begin{align}\label{eq:relalmostxu}
		\sum_{j=1}^\dN \zzpoly^{\eta + 2\xi_j}_{s} = (-1)^n \eps^{n-1} \KelvI \Dun{}^{\eta} \DLapl \KelvI(s) = 0.
	\end{align}
\end{remark}

The relations~\eqref{eq:rellindep} can be used to reduce the generating set $\mathcal{C}_n$ to a basis. For instance, the next theorem shows that if we consider only multi-indices $\fatj\in\NN^\dN$ with zero as last index, we  get a basis of the polynomial monogenics. Other bases can be constructed by following the same strategy but excluding other elements from $\mathcal{C}_n$ using the relations~\eqref{eq:rellindep}. 
The proof of the following theorem also shows that for a fixed $n\in \NN$ and $s\in\nu$, with  $\nu$ a basis of $V$, the relations~\eqref{eq:rellindep} are all independent.

\begin{theorem}\label{thm:maxwellbasis}
	Let $\nu$ be a basis of $V$, the spinor representation of $\Clif(\dN)$. The set 
	\begin{equation}
    		\base_n = \{\zzpoly^{\fatj}_s\ \mid\ \fatj = (j_1,\dots,j_{\dN-1},0)\in \NN^{\dN}, |\fatj|_1 = n, \ s\in \nu\}
	\end{equation}
	is a basis of $\Mono_n(\RR^\dN;V)$.
\end{theorem}
\begin{proof}
For every $s\in \nu$, let $s'=\eps e_d \cdot s$ and for these $s'$, consider the relations~\eqref{eq:rellindep} for all $\eta$ with $|\eta|_1=n-1$. 
These relations can be used to go from $\mathcal{C}_n$ to $\mathcal{B}_n$ by removing all polynomials $\zzpoly_{s}^{\beta}$ with $\beta_d\neq 0$. 
We will show that each relation will remove exactly one polynomial, after which the result follows by a dimension argument. We order the multi-indices $I:=\{\eta\in \NN^{\dN}\, \mid\, |\eta|_1 =n-1\}$ by reverse lexicographic order, so $\eta_1 = (0,\dots,0,n-1)$, $\eta_2 = (0,\dots,0,1,n-2), \dots, \eta_{|I|} = (n-1,0,\dots,0)$, with $ |I|=\dim\Poly_{n-1}(\RR^\dN)=\binom{n+d-2}{d-1}$.

Since $e_d\cdot s' = s$ for $s\in\nu$, we can write relation~\eqref{eq:rellindep} for $\eta_i$ and $s'\in V$ as
\begin{equation}\label{eq:relCnv2}
\zzpoly_{s}^{\eta_i+\xi_d} = - \sum_{j=1}^{d-1} \zzpoly_{e_j\cdot s'}^{\eta_i+\xi_j}. 
\end{equation}
For each $\eta_i$, we can use this to exclude the polynomials $\{\zzpoly_{s}^{\eta_i+\xi_d} \mid s \in\nu\}$, since the right-hand side of~\eqref{eq:relCnv2} is a sum of polynomials strictly lower in the ordering. Therefore, starting from the set $\mathcal{C}_n$, doing this procedure in the reverse lexicographic order for all $\eta_i$ each step will  exclude $\dim V$ polynomials from the set.  

This results in a basis, as can be seen from the dimensions of the spaces involved. Indeed, $|\mathcal{C}_n|=\dim \Poly_n(\RR^\dN)\times \dim V$ and there are $\dim \Poly_{n-1}(\RR^{\dN})\times \dim V$ different linear relations between its members; $\dim \Mono_n(\RR^\dN; V)=\dim\Poly_n(\RR^{d-1})\times \dim V$ and
 \begin{equation*}
 	 (\dim \Poly_n(\RR^\dN)- \dim\Poly_{n-1}(\RR^\dN))\times \dim V = \binom{n+\dN-2}{d-2}\times  \dim V = \dim \Mono_n(\RR^\dN;V). \qedhere
 \end{equation*}
\end{proof}


\section{Examples: the abelian cases} \label{sec:ExAb}

\subsection{Reducible reflection groups}\label{ssec:Zredu}

In this section, we consider the cases when the reflection group is reducible: $W = W_S\oplus W_T$, for $S, T$ the two root subsystems with $R$ the disjoint union of the two. 

Let $M$ be the rank of $S$. For simplicity, assume that $S$ is restricted to the first $M$ coordinates. We have an $\osp(1|2)$ realisation given by the following operators 
\begin{gather*}
\DDopM:= \sum_{a=1}^M \Dun{a}e_a, \qquad \xunM := \sum_{a=1}^M x_ae_a, \qquad \EulerM := \sum_{a=1}^M x_a\partial_{x_a},\quad
\gammaM:=  \sum_{\alpha\in S^+} \kappa(\alpha).
\end{gather*}

The odd elements $\DDopM$ and $\xunM$ generate a realisation of the superalgebra $\osp(1|2)$ with the following commutation relation
\begin{equation}
 \acomm{\DDopM}{\xunM} = 2\varepsilon \HM = 2\varepsilon(\EulerM + M/2 + \gammaM).
\end{equation}
Thus one can also define generalised symmetries in this ``sub'' $\osp(1|2)$-realisation.
\begin{definition}\label{def:partialz}
Let $1\leq j\leq M$. The partial generalised symmetry linked to $\DDopM$ is given by
	 \begin{equation}\label{eq:partialz}
	 \zM{j} := 2\varepsilon x_j\HM - \xunM\Dun{j}\xunM.
	 \end{equation}
\end{definition}

Naturally, $\zM{j}$ satisfies the equivalent relations of Proposition~\ref{prop:commZ}, Lemma~\ref{lem:zotherversions}, Proposition~\ref{prop:relofzandother} and Proposition~\ref{prop:zkelv}  since it is also in an $\osp(1|2)$ realisation. 

\subsection{The abelian cases}

We turn to the study of the abelian case, so the Dunkl--Dirac symmetry algebra for the group $W=\ZZ^\dN_2$ acting on $\RR^\dN$ with a $W$-invariant function given by the $\dN$-tuple of non-negative constants $(\kappa_1,\dots,\kappa_\dN)$. In the abelian case, the reflection $\sigma_j$ sends $x_j$ to $-x_j$ and leaves the other variables invariant. The Dunkl operators are given by
\begin{equation}
	\Dun{i} = \partial_{x_i} + \kappa_i \frac{1-\sigma_i}{x_i},
\end{equation}
and the commutation relation~\eqref{eq:commDx} becomes
\begin{equation}
	\comm{\Dun{i}}{x_j} = \delta_{ij} (1+2\kappa_i\sigma_i).
\end{equation}

Albeit Theorem~\ref{thm:maxwellbasis} gives a basis of $\Mono_n$, the specificity of the group studied call for a slightly different approach. The complete reducibility of $W=\ZZ_2^\dN$ was used in~\cite{de_bie_z2n_2016} to construct a basis from the Cauchy--Kovalevskaya extension Theorem. We will retrieve this construction from the generalised symmetries of Definition~\ref{def:partialz}.

\subsection{The Cauchy--Kovalevskaya basis}

In the abelian case, there exists a generalisation of the Cauchy-Kovalevskaya map. It can be used to construct a basis of the polynomial monogenics.

\begin{proposition}[{\cite[Eq.~(31)]{de_bie_z2n_2016}}]
	Let $V$ be an irreducible representation of the Clifford algebra $\Clif(\dN)$. There is an isomorphism between the space of spinor-valued polynomials of degree $n$ over $k-1$ variables and the monogenics of degree $n$ over $k$ variables given by 
	\begin{equation}\label{eq:CK}
		\begin{gathered}
			\CK_{x_k}^{\kappa_k}: \Poly_{n}(\RR^{k-1})\otimes V \longrightarrow \Mono_{n}(\RR^k;V)\\
			p\mapsto \CK_{x_k}^{\kappa_k}(p) = \sum_{a=0}^{\lfloor n/2\rfloor} \frac{x_k^{2a} \DDopA{k-1}^{2a}}{2^{2a}a!(\kappa_k+1/2)_a}p - \varepsilon \frac{e_kx_k\DDopA{k-1}}{2}\sum_{a=0}^{\lfloor \frac{n-1}{2}\rfloor} \frac{x_k^{2a}\DDopA{k-1}^{2a}}{2^{2a}a!(\kappa_k+1/2)_{a+1}}p.
		\end{gathered}
	\end{equation}
\end{proposition}
Note that in~\cite{de_bie_z2n_2016}, the proposition is given for $\varepsilon = -1$. The proof for the two signs is the same up to minor modifications.

The map $\CK_{x_k}^{\kappa_k}$ is an isomorphism and has an inverse given by the map evaluating the last variable to $0$:
\begin{equation}
	\begin{gathered}
		R_k: \Mono_{n}(\RR^k;V) \longrightarrow \Poly_n(\RR^{k-1})\otimes V\\
		f(x_1,\dots , x_k) \longmapsto R_k(f)(x_1,\dots,x_{k-1}) = f(x_1,\dots , x_{k-1},0).
	\end{gathered}
\end{equation}

Consider now the Fischer decomposition of the polynomial space.
\begin{proposition}[Fischer decomposition~\cite{brackx1984clifford}]
	When $\kappa$ is positive, the space of spinor-valued polynomials decomposes in monogenic spaces as
	\begin{equation}
		\Poly_n(\RR^\dN)\otimes V = \bigoplus_{k=0}^n \xunA{\dN}^{n-k}\Mono_k(\RR^\dN;V).
	\end{equation}
\end{proposition}

From this proposition and the tower of CK extensions and Fischer decompositions (see for example~\cite[Eq.~(32)]{de_bie_z2n_2016}), we get a basis of the space of monogenics. 
\begin{proposition}[{\cite[Prop.~6]{de_bie_z2n_2016}}]\label{prop:basisCK}
	Let $\{s\}_{s\in \nu}$ be a basis of the spinor representation $V$. The set of functions defined for all $\fatj = (j_1,\dots,j_{\dN-1},0) \in \NN^\dN$ with $|\fatj|_1 =n$,
	\begin{equation}
		\Psi^{\fatj}_s(x_1,\dots, x_\dN) = \CK_{x_\dN}^{\kappa_\dN} \left( \xun_{\dN-1}^{j_{\dN-1}} \CK_{x_{\dN-1}}^{\kappa_{\dN-1}} \left( \dots \CK_{x_3}^{\kappa_3} \left( \xun_{2}^{j_2} \CK_{x_2}^{\kappa_2}(x_1^{j_1})\right)\dots \right)\right) s,
	\end{equation}
	is a basis of $\Mono_n(\RR^\dN;V)$.
\end{proposition}

\subsection{A new basis}
We will use the partial  generalised symmetries of Subsection~\ref{ssec:Zredu} to make full use of the completely reducible nature of $\ZZ^\dN_2$. The crucial point of the abelian case $W=\ZZ_2^\dN$ is a chain of inclusions
\begin{equation}
	\ZZ_2 \subset \ZZ_2^2 \subset \dots \subset \ZZ_2^{\dN-1} \subset \ZZ_2^\dN.
\end{equation}
This gives in turn a tower of $\osp(1|2)$ algebra realisations given by the pairs $(\DDopA{k},\xunA{k})$ for each $1\leq k\leq \dN$. This feature of the group $\ZZ_2^\dN$ was used in Proposition~\ref{prop:basisCK} to give a basis. We give a basis proportional to the CK basis by replacing the operators $\zz_j$ in Theorem~\ref{thm:maxwellbasis} by the partial ones $\zzA{j}$. This is done by linking the Cauchy--Kovalevskaya extension of each level to one partial generalised symmetries.

An important note, the commutation of the $\zzA{j}$ requires that they stay on the same level. Indeed, two partial generalised symmetries at a different level in the tower do not commute in general. So in the basis of the following proposition, the order of application matters.

\begin{proposition}\label{prop:partialzbasis}
	The set of polynomials of the form
	\begin{equation}
		\Phi^{\fatj}_{s} := \zzA{\dN}^{j_{\dN-1}} \zzA{\dN-1}^{j_{\dN-2}} \dots \zzA{2}^{j_1} s, \quad \text{for  }\fatj= (j_1,\dots j_{\dN-1},0)\in \NN^{\dN},\ |\fatj|_1 =n, \text{ and } s\in\nu
	\end{equation}
	constitutes a basis of $\Mono_{n}(\RR^{\dN};V)$.
\end{proposition}

 The proof of this proposition will follow from Proposition~\ref{prop:maxwellandck}, as it exhibits a change of basis from the $\Psi_s^{\fatj}$ to the $\Phi^{\fatj}_s$. The remaining of the section is dedicated to proving this change of basis.
 
  Propositions~\ref{prop:z2propCK} and~\ref{prop:zkpropCK} show that, as operators, $\zzA{k}^j$ and $\CK_{x_k}^{\kappa_k} \xunA{k-1}^j$ are Clifford proportional, meaning that they differ only by a Clifford number. There is a small difference between $k=2$ and $k>2$ and thus we separate the proof in two steps. We begin by showing what will constitute the hard part of the induction proof of Proposition~\ref{prop:z2propCK}.

\begin{lemma}\label{lem:stepinduction}
	The $\zzA 2$ operator and the $\CK_{x_2}^{\kappa_2}$ extension are linked by
	\begin{equation}\label{eq:lem:stepinduction}
		\zzA{2} \CK_{x_2}^{\kappa_2} (x_1^m s) = A_m \CK_{x_2}^{\kappa_2} (x_1^{m+1}e_2e_1 s),
	\end{equation}
	with
	\begin{equation}
		A_m = 1+m+(1-(-1)^m)\kappa_1 + 2\kappa_2.
	\end{equation}
\end{lemma}

\begin{proof}
	Recall that $R_2$ is the inverse of $\CK_{x_2}^{\kappa_2}$. Acting with $R_2$ on~\eqref{eq:lem:stepinduction} thus yields
		\begin{equation}\label{eq:lhsoflem}
			R_2 \zzA{2} \CK_{x_2}^{\kappa_2}(x_1^m s) = A_m x_1^{m+1}e_2e_1s,
		\end{equation}
		so it suffices to compute the left-hand side of~\eqref{eq:lhsoflem}. Begin by using the expression~\eqref{eq:lem:zotherversions3} of $\zzA{2}$
	\begin{align*}
		R_2 (\zzA{2} \CK_{x_2}^{\kappa_2}( x_1^m  s)) &= R_2( (2\varepsilon x_2 (\EulerA 2 +1 + \gamma_2) - \xun_2(1+2\kappa_2\sigma_2)e_2 -\varepsilon \xsq_2 \Dun{2})\CK_{x_2}^{\kappa_2}( x_1^m  s)).
		\intertext{Since $R_2$ sends $x_2$ to $0$, this reduces to} 
		&= (-x_1e_1(1+2\kappa_2\sigma_2)e_2)x_1^m s +\varepsilon^2 x_1^2 \Dun{2}\frac{x_2e_2 \Dun{1}e_1}{2}\frac{x_1^m s}{(1/2+\kappa_2)}\\
		&= ((1+2\kappa_2\sigma_2)e_2e_1)x_1^{m+1} s + x_1^2 (1+2\kappa_2\sigma_2) \frac{\Dun{1}x_1^m e_2e_1 s}{(1+2\kappa_2)},
		\intertext{and now we apply $\Dun{1}x_1^m s = (m + \kappa_1(1-(-1)^m))x_1^{m-1}  s$, since $\Dun{1} s =0$, to obtain}
		&= (1+ m + 2\kappa_2  +(1-(-1)^m)\kappa_1) x_1^{m+1}e_2e_1  s. \qedhere
	\end{align*}
\end{proof}
Using this lemma, we can prove the general proposition.
\begin{proposition}\label{prop:z2propCK}
	Acting on a spinor $ s$, we have
	\begin{equation}
		\zzA2^j s = a_{2}^j \CK_{x_2}^{\kappa_2}(x_1^j (e_2e_1)^j s),
	\end{equation}
	with
	\begin{equation}
		a_{2}^j:=2^j(\kappa_2+1/2)_{\lfloor (j+1)/2\rfloor} (\gamma_2+1)_{\lfloor j/2\rfloor}.
	\end{equation}
\end{proposition}
\begin{proof}
	We proceed by induction on $j$, the case $j=1$ being covered by Lemma~\ref{lem:stepinduction} with $m=0$. Assume the induction hypothesis holds up to $j=m$. Now we consider the $(m+1)$th step and apply the induction hypothesis
	\begin{align}
		\zzA 2^{m+1} s &= \zzA 2\zzA 2^m s = \zzA 2 a_2^m\CK_{x_2}^{\kappa_2}( x_1^m (e_2e_1)^m  s),
	\end{align}
	then we apply Lemma~\ref{lem:stepinduction} to get $\zzA 2^{m+1} s = a_2^{m+1}\CK_{x_2}^{\kappa_2}( x_1^{m+1} (e_2e_1)^{m+1} s)$  since $A_ma_2^m = a_{2}^{m+1}$.
\end{proof}

In general, for $k>2$, there is one additional difficulty: the CK map includes not only Dunkl derivatives, but also partial Dunkl--Dirac operators. We will thus need a small lemma. 
\begin{lemma}[{\cite[Lem.~13]{de_bie_z2n_2016}}]\label{lem:DDAonxunAn}
	Let $f\in \Mono_n(\RR^k;V)$. The action of $\DDopA{k}$ on $\xunA{k}^mf$ is given by
	\begin{equation}
		\DDopA{k}(\xunA{k}^m f) = \eps(m+\frac{(1-(-1)^m)}{2} (2n + k-1 + 2\gamma_k))\xunA{k}^{m-1}f.
	\end{equation}
\end{lemma}
\begin{proof}
This follows by induction from the relations $\comm{\DDopA{k}}{\xsq_k} = 2\xsq_k$, $\acomm{\DDopA{k}}{\xunA{k}}= 2\varepsilon(\EulerA k+k/2+\gamma_k)$, $\DDopA{k}f =0$ and $\EulerA k f = nf$.
\end{proof}

Now to prove the relation between the partial generalised symmetry and the CK map for the other levels of the tower, we introduce a lemma that takes care of the difficult induction step.

\begin{lemma}\label{lem:inductionstepgen}
	Let $f\in\Mono_{n}(\RR^{k-1};V)$ be a monogenic of degree $n$ in the first $k-1$ variables. Then
	\begin{equation}\label{eq:zkCKxmf}
		\zzA{k} \CK_{x_k}^{\kappa_k} (\xunA{k-1}^m f) = B_{k,n}^m \CK_{x_k}^{\kappa_k} (\xunA{k-1}^{m+1}e_kf),
	\end{equation}
	with
	\begin{equation}
		B_{k,n}^m =  (-1)^{m+1}(m +1+ \frac{(1-(-1)^m)}{2} (2n+k-2 + 2\gamma_{k-1})+2\kappa_k ). 
	\end{equation}
	
\end{lemma}
\begin{proof}
	The proof proceeds in the same fashion as the one of Lemma~\ref{lem:stepinduction}, using $\EulerA {k-1} f = n f$, $\DDop_{k-1} f = 0$ and $\Dun{k}f=0$ instead of $\Euler s= 0$, $\Dun{1}  s =0$ and $\Dun{2}  s =0$ in the corresponding steps.
	
	The map $\CK_{x_k}^{\kappa_k}$ has an inverse $R_k$ defined as evaluating $x_k$ to $0$. We compute, using~\eqref{eq:CK},
	\begin{align*}
		R_k(\zzA{k} \CK_{x_k}^{\kappa_k} (\xunA{k-1}^mf)) &= R_k(2\varepsilon x_k (\EulerA k + k/2 + \gamma_k) - \xunA{k}(1+2\kappa_k\sigma_k)e_k - \eps \xsq_k \Dun{k})\CK_{x_k}^{\kappa_k}(\xunA{k-1}^m f)\\
		&= -\xunA{k-1}(1+2\kappa_k\sigma_k)e_k\xunA{k-1}^mf + \eps^2 \xsq_{k-1} \Dun{k} \frac{e_kx_k\DDopA{k-1}}{2(\kappa_k +1/2)}\xunA{k-1}^nf\\
		&=(-1)^{m+1}(1+2\kappa_k)\xunA{k-1}^{m+1} e_k f + \xsq_{k-1} e_k \frac{\comm{\Dun{k}}{x_k}}{2(\kappa_k+1/2)} \DDopA{k-1}\xunA{k-1}^nf
		\intertext{and we use Lemma~\ref{lem:DDAonxunAn} for $f\in\Mono_n(\RR^{k-1};V)$ on the rightmost term to get}
		&= (-1)^{m+1}(1+2\kappa_k)\xunA{k-1}^{m+1} e_k f \\
		&\quad+ 
		\eps \Big(m + \frac{(1-(-1)^m)}{2} (2n+k-2 + 2\gamma_{k-1})\Big) \xsq_{k-1} e_k \xunA{k-1}^{m-1}f \\
		&= (-1)^{m+1}(1+2\kappa_k+ \eps^2 (m + \frac{(1-(-1)^m)}{2} (2n+k-2 + 2\gamma_{k-1})))   \xunA{k-1}^{m+1}e_kf.
	\end{align*}
	This allows one to determine the constant $B_{k,n}^m$ by comparing with~\eqref{eq:zkCKxmf}.
\end{proof}
\begin{proposition}\label{prop:zkpropCK}
	Let $f\in\Mono_{n}(\RR^{k-1};V)$ be a monogenic in $k-1$ variables of degree $n$. For $k>2$,
	\begin{equation}
		\zzA{k}^j f = b_{k,n}^j \CK_{x_k}^{\kappa_k} (\xunA{k-1}^j e_k^jf),
	\end{equation}
	with 
	\begin{equation}
		b_{k,n}^j=(-1)^{\lfloor (j+1)/2\rfloor} 2^j (\kappa_k+1/2)_{\lfloor  (j+1)/2\rfloor} (\gamma_{k} + n + k/2)_{\lfloor j/2\rfloor}.
	\end{equation}
\end{proposition}
\begin{proof}
	We proceed by induction on $j$. The base case follows from Lemma~\ref{lem:inductionstepgen} with $m=0$.
	Assume the induction hypothesis holds up to $j$. The induction step follows from the induction hypothesis and Lemma~\ref{lem:inductionstepgen}
	\begin{align}
		\zzA k^{j+1}f &= \zzA k\zzA k^jf = \zzA k b_{k,n}^j\CK_{x_k}^{\kappa_k} (\xun_{k-1}^j e_k^j f) = B_{k,n}^j b_{k,n}^j\CK_{x_k}^{\kappa_k}(\xun_{k-1}^{j+1} e_k^{j+1}f).
	\end{align}
	This shows the result since $B_{k,n}^jb_{k,n}^j = b_{k,n}^{j+1}$.
\end{proof}

Connecting this to the CK basis, we get the following correspondence,  proving Proposition~\ref{prop:partialzbasis}.

\begin{proposition}\label{prop:maxwellandck}
	Let $\fatj = (j_1,\dots, j_{\dN-1},0) \in \NN^{\dN}$ with $|\fatj|_1 = n$ and $s$ be a spinor. The partial generalised symmetry basis is linked to the CK basis by
	\begin{equation}
		\Phi^{\fatj}_s = c_{\fatj}\Psi^{\fatj}_{\fatj\cdot s},
	\end{equation}
	where the action $\fatj\cdot s$ in $\Psi^{\fatj}_{\fatj\cdot s}$ denotes the action $e_{\dN}^{j_{\dN-1}} \dots e_{3}^{j_2} (e_2e_1)^{j_1}s$ on the spinor space in the expression of the polynomial $\Psi$ and where the proportionality constant is given by
	\begin{equation}
		\begin{aligned}
			c_{\fatj}&= \left(\prod_{k=3}^{\dN-1}\prod_{l=2}^{k-1} (-1)^{j_kj_l}\right)2^n(1/2+\kappa_2)_{\lfloor(j_1+1)/2\rfloor}(1+\gamma_2)_{\lfloor j_1/2\rfloor} \times \\ & \quad \prod_{i=2}^{\dN-1}((-1)^{\lfloor (j_i+1)/2\rfloor} (1/2+\kappa_{i+1})_{\lfloor(j_i+1)/2\rfloor} ((i+1)/2  +\gamma_{i+1}+ \sum_{k=1}^{i-1}j_{k})_{\lfloor j_i/2\rfloor}).
		\end{aligned}
	\end{equation}
\end{proposition}
\begin{proof}
	The first steps are to apply once Proposition~\ref{prop:z2propCK} for the $\zzA2$ contribution and multiple times Proposition~\ref{prop:zkpropCK} for the remaining contributions of the $\zzA{k}$. This will give $c_{\fatj}$ up to the first sign. This sign is obtained when Clifford elements go to the right from their interaction with the vector variables. Note that $e_k$ commutes with $\CK_{x_l}^{\kappa_l}$ when $k>l$, as can be clearly seen from the expression~\eqref{eq:CK}, so the only sign to consider is from the vector variable crossing. Step by step, this gives
	\begin{align*}
		\Phi_{\fatj}^s &= \zzA{\dN}^{j_{\dN-1}} \zzA{\dN-1}^{j_{\dN-2}} \dots \zzA{2}^{j_1}  s\\
		\text{\scriptsize(Prop.~\ref{prop:z2propCK})}\quad &= a_2^{j_1}\zzA{\dN}^{j_{\dN-1}} \zzA{\dN-1}^{j_{\dN-2}} \dots \zzA{3}^{j_2} \CK_{x_2}^{\kappa_2}(x_1^{j_1}(e_2e_1)^{j_1} s)\\
		\text{\scriptsize(Prop.~\ref{prop:zkpropCK})}\quad &= a_2^{j_1} \prod_{k=2}^{\dN-1} (b_{k+1,\sum_{j=1}^{k-1}j_k}^{j_k}) \CK_{x_\dN}^{\kappa_\dN}(\xunA{\dN-1}^{j_\dN}e_\dN^{j_{\dN-1}}\CK_{x_{\dN-1}}^{\kappa_{\dN-1}}(\cdots e_3^{j_2}\CK_{x_2}^{\kappa_2}(x_1^{j_1}(e_2e_1)^{j_1}  s)) )\\
		&= \left(\prod_{k=3}^{\dN-1}\prod_{l=2}^{k-1} (-1)^{j_kj_l}\right)a_2^{j_1} \prod_{k=2}^{\dN-1} (b_{k+1,\sum_{j=1}^{k-1}j_k}^{j_k}) \CK_{x_\dN}^{\kappa_\dN}(\xunA{\dN-1}^{j_\dN}\dots \CK_{x_2}^{\kappa_2}(x_1^{j_1}e_\dN^{j_{\dN-1}}\dots (e_2e_1)^{j_1}  s) )\\
		&=c_{\fatj} \Psi_{\fatj}^{\fatj\cdot s}.\qedhere
	\end{align*}
\end{proof}

%
%

\section*{Acknowledgements}
We thank the anonymous referees for their comments. This project was supported in part by the EOS Research Project [grant number 30889451]. Moreover, ALR holds a scholarship from the Fonds de recherche du Qu\'ebec -- Nature et technologies [grant number 270527], and RO was supported by a postdoctoral fellowship, fundamental research, of the Research Foundation -- Flanders (FWO) [grant number 12Z9920N].

\begingroup 
\bibliographystyle{abbrv}
\bibliography{2021_DBLROVdJ_maxwell}
\endgroup

\end{document}